\documentclass[leqno,a4paper,12pt]{amsart}
\usepackage{amsmath,amsthm,amssymb,amsfonts,enumerate,color}
\allowdisplaybreaks[4]

\newtheorem{thm}{Theorem}[section]
\newtheorem{prop}[thm]{Proposition}

\newtheorem{lem}[thm]{Lemma}
\newtheorem{defi}[thm]{Definition}
\newtheorem{remark}[thm]{Remark}
\newtheorem{example}[thm]{Example}
\newtheorem{pb}[thm]{Problem}

\newtheorem{theoA}{Theorem}

\newtheorem{remA}[theoA]{Remark}

\numberwithin{equation}{section}
\newcommand{\R}{{\mathbb R}}

\newcommand{\real}{{\mathbb R}}

\newcommand{\ent}{{\mathbb Z}}

\newcommand{\norm}[1]{\left\Vert#1\right\Vert}
\newcommand{\abs}[1]{\left\vert#1\right\vert}

\newcommand{\M}{{\mathcal M}}

\newcommand{\8}{\infty}

\begin{document}

\title[Differential transforms for heat semigroups]
{Boundedness of differential transforms for heat semigroups generated by fractional Laplacian}

\thanks{{\it 2020 Mathematics Subject Classification:} 42B20, 42B25.}
\thanks{{\it Key words:} differential transforms,  heat semigroup, fractional Laplacian, maximal operator, lacunary sequence.}

\author{Xinyu Ren}

 \address{School of Statistics and Mathematics \\
             Zhejiang Gongshang University \\
             Hangzhou 310018, P.R. China}
 \email{renxinyuanhui@163.com}

\author{ Chao Zhang}

 \address{School of Statistics and Mathematics \\
             Zhejiang Gongshang University \\
             Hangzhou 310018, P.R. China}
 \email{zaoyangzhangchao@163.com}

\thanks{Chao Zhang is supported by the National Natural Science Foundation of China(Grant No. 11971431), the Zhejiang Provincial Natural Science Foundation of China(Grant No. LY18A010006), and the first Class Discipline of Zhejiang-A(Zhejiang Gongshang University-Statistics.}

\date{}
\maketitle

\begin{abstract}
 In this paper we analyze the convergence of the following type of series
\begin{equation*}  T_N f(x)=\sum_{j=N_1}^{N_2} v_j\Big(e^{-a_{j+1}(-\Delta)^\alpha} f(x)-e^{-a_{j}(-\Delta)^\alpha} f(x)\Big),\quad x\in \mathbb R^n,
\end{equation*}
where  $\{e^{-t(-\Delta)^\alpha} \}_{t>0}$ is the heat semigroup of the fractional Laplacian $(-\Delta)^\alpha,$  $N=(N_1, N_2)\in \mathbb Z^2$ with $N_1<N_2,$ $\{v_j\}_{j\in \mathbb Z}$ is a bounded real sequences  and $\{a_j\}_{j\in \mathbb Z}$ is an increasing real sequence.
Our analysis will consist in the boundedness, in  $L^p(\mathbb{R}^n)$ and in  $BMO(\mathbb{R}^n)$,  of the operators
$T_N$ and its maximal operator $\displaystyle T^*f(x)= \sup_N |T_N f(x)|.$
It is also shown that the local size of the maximal differential transform operators  is the same with the order  of a singular integral for functions $f$ having local support.
 \end{abstract}


\vskip 1cm

\section{Introduction} \label{Sec:L2}

Let $\displaystyle \Delta = \sum_{j=1}^n\frac{\partial^2}{\partial x^2_j}$ be the Laplace operator in $\mathbb{R}^n$. Its heat semigroup is defined by
 \begin{align*}\label{Formu:HeatPointwise}
e^{t\Delta}\varphi(x)=\int_{\real^n} e^{t\Delta}(x-y)\varphi(y)dy,~x\in \real^{n},\ \ t>0,
\end{align*}
where $e^{t\Delta}(x)$ denotes the Gauss-Weierstrass kernel
  $$e^{t\Delta}(x)=\frac{1}{(4\pi t)^{n/2}} e^{-\frac{|x|^2}{4t}}.$$
And the  fractional Laplacian can be defined as a pseudo-differential operator via the Fourier  transform
$$ \mathcal{F}((-\Delta)^\alpha f)(\xi)=|\xi|^{2\alpha}\mathcal F( f) (\xi),$$ where $\mathcal F$ is the Fourier transform. The corresponding fractional heat semigroup is defined as
\begin{equation*}
 \mathcal F\left(e^{-t(-\Delta)^\alpha}f\right) (\xi):=e^{-t|\xi|^{2\alpha}}\mathcal F (f)(\xi), \quad  \alpha\in (0, 1).
\end{equation*}
When $\alpha=1/2$, it is the Poisson semigroup.
In the literature, the fractional heat semigroup $\{e^{- {t(-\Delta)^\alpha}}\}_{t>0}$ has widely used in the study of partial differential equations, harmonic analysis, potential theory and modern probability theory.  For example,
the semigroup $\{e^{- {t(-\Delta)^\alpha}}\}_{t>0}$  is usually applied to construct the linear part of solutions to fluid equations in the mathematic physics, e.g. the generalized Naiver-Stokes equation, the quasi-geostrophic equation, the
MHD equations. In fact, $e^{- {t(-\Delta)^\alpha}}f(x)$ is the solution of the heat equation related to the fractional Laplacian:
 \begin{align*}
\left\{
 \begin{array}{rrl}
  \partial_tu(x,t)+(-\Delta)^\alpha u(x,t)&=&0,\ \ \ \quad   (x,t)\in   \real_+^{n+1},\\
  u(x,0)&=&f(x), \ \ x\in \real^n.
 \end{array}
 \right.
 \end{align*}
 And, in the field of probability theory, the researchers use $\{e^{- {t(-\Delta)^\alpha}}\}_{t>0}$  to describe some kind
of Markov processes with jumps. For further information and the related applications of fractional heat
semigroups $\{e^{- {t(-\Delta)^\alpha}}\}_{t>0}$ , we refer the reader to \cite{CX, Grigo}. In \cite{MYZ}, by an invariant derivative technique and the Fourier analysis method, the authors concluded that the kernel, $ e^{- {t(-\Delta)^\alpha}}(x) $ satisfy the following pointwise estimate
$$0<{e^{- {t(-\Delta)^\alpha}}(x)}\le \frac t{(t^{1/2\alpha}+|x|)^{n+2\alpha}}, \quad (x,t)\in \real_+^{n+1}.$$

In this article, we will introduce the heat semigroup $\{e^{- {t(-\Delta)^\alpha}}\}_{t>0}$ into the analysis of martingale transforms in probability. Martingale transforms was considered firstly by D. L. Burkholder in 1966; see \cite{Burk}. In \cite{Burk}, the author proved the almost everywhere convergence of the martingale transforms. In martingale theory, we always treat the martingale transforms as a corresponding tool of the singular integral operators in harmonic analysis. In fact, we want to analyze the behavior of the following type sum
\begin{equation}\label{Formu:SquareFun}
 \sum_{j\in \ent} v_j(e^{-a_{j+1}(-\Delta)^\alpha} f(x)-e^{-a_{j}(-\Delta)^\alpha} f(x))
\end{equation}
where $\{v_j\}_{j\in \ent}$ is a bounded sequence of real numbers and  $\{a_j\}_{j\in \ent}$ is  an increasing  sequence of positive numbers. Observe that in the case $v_j\equiv1$, the above series is  telescopy, and their behavior coincide with $e^{-t(-\Delta)^\alpha} f(x)$. This way of analyzing convergence of sequences was considered by Jones and Rosemblatt for ergodic averages(see \cite{JR}), and latter by Bernardis et al. for differential transforms(see \cite{BLMMDT}). The authors considered the differential transforms related to the one-sided fractional Poisson type operator sequence and the heat semigroup generated by Laplacian(see \cite{ZMT, ZT}).

To better understand the behavior of the sum in (\ref{Formu:SquareFun}), we shall analyze its ``partial sums'' defined as follows.
For each $N\in \ent^2,~N=(N_1,N_2)$ with $N_1<N_2,$ we define the partial sum operators
\begin{equation}\label{Formu:FinSquareFun}
 T_N f(x)=\sum_{j=N_1}^{N_2} v_j(e^{-a_{j+1}(-\Delta)^\alpha} f(x)-e^{-a_{j}(-\Delta)^\alpha} f(x)),\ x\in \mathbb R^n.
\end{equation}
 We shall also consider the   maximal operator
\begin{equation}\label{Formu:MaxSquareFun}
 T^*f(x)=\sup_N \abs{T_N f(x)}, \quad x\in\real^n,
 \end{equation}
where the supremum are taken over all $N=(N_1,N_2)\in \ent^2$ with $N_1< N_2$.
In \cite{ZMT}, the authors proved the boundedness of the above operators related with the one-sided fractional Poisson type operator sequence. And the same results was gotten for the above operators related with the heat semigroup generated by Laplacian in \cite{ZT}.

Some of our results will be valid only when the sequence  $\{a_j\}_{j\in \mathbb Z}$ is lacunary.  That means that  there exists a $\lambda >1$ such that $\displaystyle \frac{a_{j+1}}{a_j} \ge \lambda, \, j \in \mathbb{Z}$.  In particular, we shall prove  the boundedness of the operator $T^*$ in the weighted spaces
$L^p(\mathbb R^{n}, \omega),$ where $\omega$ is  the usual Muckenhoupt weight on $\mathbb R^{n}$. We refer the reader to the book by J. Duoandikoetxea \cite[Chapter 7]{Duo} for definitions and properties of the $A_p$ classes.
We shall also analyze the boundedness behavior of the operators in $L^\infty$ and  $BMO$  spaces. The space $BMO(\real^n)$ is defined as the set of functions $f$ such that $f^{\sharp}  \in L^\infty(\mathbb{R}^n)$, where
$$f^{\sharp} (x) = \sup_{x \in B} \Big\{ \frac1{|B|} \int_B\Big|f(z)- \frac1{|B|} \int_B f \Big| dz    \Big\}.$$
We define  $\|f\|_{BMO(\R^n)}= \|f^{\sharp}  \|_{L^\infty(\mathbb{R}^n)}$.  In fact, we have the following results:
\begin{thm}\label{Thm:PoissonLp}    Assume that the sequence $\{a_j\}_{j\in \mathbb Z}$ is a $\lambda$-lacunary sequence with $\lambda >1$. Let $T^*$ be  the operator defined in \eqref{Formu:MaxSquareFun}. Then
\begin{enumerate}[(a)]
    \item for any $1<p<\infty$ and $\omega\in A_p$,  there exists a constant $C$ depending  on $n, p,\omega,  \lambda$, $\norm{v}_{l^\infty(\mathbb Z)}$ and $\alpha$ such that
 $$\norm{T^* f}_{L^p(\mathbb R^{n}, \omega)}\leq C\norm{f}_{L^p(\mathbb R^{n}, \omega)},$$
 for all functions $f\in L^p(\real^{n}, \omega).$
    \item for any  $\omega\in A_1$, there exists a constant $C$ depending  on $n, \omega, \lambda$,  $\norm{v}_{\ell^\infty(\mathbb Z)}$  and  $\alpha$  such that
 $$\omega\left({\{x\in \real^{n}:\abs{T^* f(x)}>\sigma\}}\right) \le C\frac{1}{\sigma}\norm{f}_{L^1(\mathbb R^{n}, \omega)}, \quad \sigma>0,$$
for all functions $f\in L^1(\real^{n}, \omega).$

\item  given $f\in L^\infty(\real^n),$ then either $T^* f(x) =\infty$ for all $x\in \mathbb R^n$, or $T^* f(x) < \infty$ for $a. e.$  $x\in \mathbb R^n$. And in this latter case, there exists a constant $C$ depending  on $n, \lambda$,  $\norm{v}_{\ell^\infty(\mathbb Z)}$ and $\alpha$ such that
$$\norm{T^*f}_{BMO(\mathbb R^n)}\leq C\norm{f}_{L^\infty(\mathbb R^n)}.$$
\item  given $f\in BMO(\real^n),$ then either $T^* f(x) =\infty$ for all $x\in \mathbb R^n$, or $T^* f(x) < \infty$ for $a. e.$  $x\in \mathbb R^n$.  And in this latter case,  there exists a constant $C$ depending  on $n, \lambda$, $\norm{v}_{\ell^\infty(\mathbb Z)}$ and $\alpha$ such that
\begin{equation}\label{sharp}\norm{T^*f}_{ BMO(\mathbb R^n)}\leq C\norm{f}_{BMO(\mathbb R^n)}.
\end{equation}
\end{enumerate}
The constants $C$ appeared above all are independent of $N.$
\end{thm}

\begin{remark}
  From the conclusions  in Theorem \ref{Thm:PoissonLp}, for $f\in L^p(\real^{n}, \omega)$ with $\omega\in A_p$,  we can define $T f$ by the limit of $T_N f$ in $L^p$ norm
  $$T f(x)=\lim_{(N_1,N_2)\rightarrow (-\infty, +\infty)} T_N f(x),\quad \quad ~x\in \real^{n}.$$
 For more results related with the convergence of  $T_N f$, see Theorem \ref{Thm:ae}.
\end{remark}

In classical harmonic analysis,  if $f= \chi_{(0,1)}$ and $\mathcal{H}$ is the Hilbert transform, it is known  that   $\displaystyle  \frac1{r} \int_{-r}^0 \mathcal{H}(f)(x)dx\  \sim \log\frac{e}{r}$  as $ r \to 0^+$. In general, this is the growth of a singular integral applied to a bounded function  at the origin. The following theorem shows that if $f$ is a  bounded function, the growth of  $T^*f$ at the origin   is of the same order of a singular integral operator.
Some related results  about the local behavior of variation operators can be found in \cite{BCT}. One dimensional results about the variation of  convolution operators can be found in \cite{MTX}. And  one dimensional results about  differential transforms of one-sided fractional Poisson type operator sequence is proved in \cite{ZMT}.

The following theorem analyzes the local growth behavior of $T^*$  in $L^\infty$:

\begin{thm} \label{Thm:GrothLinfinity}
Let    $\{v_j\}_{j\in \mathbb Z}\in l^p(\mathbb Z)$ for some $1 \le p\le \infty.$
  Let $\{a_j\}_{j\in \mathbb Z}$ be any increasing sequence and  $T^*$  defined in (\ref{Formu:MaxSquareFun}). Then for every $f\in L^\infty(\mathbb{R}^n)$ with support in the unit ball $B=B(0, 1)$,  for any ball $B_r\subset B$ with $2r<1$, we have
    $$\frac{1}{|B_r|} \int_{B_r} \abs{T^* f (x)} dx\leq C\left(\log \frac{2}{r}\right)^{1/p'}\norm{v}_{l^p(\mathbb Z)}\|f\|_{L^\infty(\mathbb R^n)}.$$
In the statement above,  $\displaystyle p' = \frac{p}{p-1},$ and if $p=1$, $\displaystyle p'=\infty.$
\end{thm}

This article is organized as follows. In Section \ref{Sec:Laplacian}, by using Calder\'on-Zygmund theory,  we prove the uniform boundedness of the  operators $T_N.$ In Section \ref{Sec:maxi}, we give the proof of Theorem \ref{Thm:PoissonLp}.  And we prove Theorem \ref{Thm:GrothLinfinity} in the last section.

Throughout this article, the letters $C, c$ will denote  positive
constants which may change from one instance to another and depend on
the parameters involved. We will make a frequent use, without
mentioning it in relevant places, of the fact that for a positive
$A$ and a non-negative $a,$
$$\sup\limits_{t>0}t^a
e^{-At}=C_{a, A}<\infty.$$




\vskip 1cm


\section{Uniform $L_p$ boundedness of the operators $T_N$}\label{Sec:Laplacian}

In this section, we  will make some preparations to prove Theorem \ref{Thm:PoissonLp}.  In fact, we will prove the uniform boundedness of the operators $T_N.$ {The standard  Calder\'on-Zygmund theory will be a fundamental tool  in proving the $L^p$ boundedness of the operators  $T_N$. For this theory, the reader can see  some classical textbooks about harmonic analysis, for example, see \cite{Duo, Grafakos}. Nowadays it is well known that the fundamental ingredients in the theory are the $L^{p_0}(\mathbb{R}^n)$ boundedness for some $1<p_0\le \infty$ and the smoothness of the kernel of the operator. Even more, the constants that appear in the results only depend on the boundedness constant in $L^{p_0}(\mathbb{R}^n)$ and the constants related with the size and smoothness of the kernel.}

  In the following proposition, we present and prove  the $L^2$ boundedness of the operators $T_N$.
\begin{prop}\label{Thm:L2Estimate}
There is a constant $C>0$ depending  on $n$ and $\norm{v}_{\ell^\infty(\mathbb Z)}$  (not on $N$)  such that
 $$ \|T_N f \|_{L^2(\real^{n})}\leq C \|f \|_{L^2(\real^{n})}.$$
\end{prop}
\begin{proof}
Let $f\in L^2(\real^{n})$. Using  the   Plancherel  theorem, we have
\begin{align*}
 \norm{T_N f }^2_{L^2(\real^{n})} & = \norm{\sum_{j=N_1}^{N_2} v_j\left(e^{-a_{j+1}(-\Delta)^\alpha} f -e^{-a_{j}(-\Delta)^\alpha}f\right)}^2_{L^2(\real^{n})}\\
 &=  \int_{\mathbb{R}^n} \Big\{\sum_{j=N_1}^{N_2} v_j\left(e^{-a_{j+1}(-\Delta)^\alpha} f(x) -e^{-a_{j}(-\Delta)^\alpha}f(x)\right) \Big\}^2 dx \\
  &=
 \int_{\mathbb{R}^n} \left(\mathcal F{\Big\{\sum_{j=N_1}^{N_2} v_j\left(e^{-a_{j+1}(-\Delta)^\alpha} f(\cdot) -e^{-a_{j}(-\Delta)^\alpha}f(\cdot)\right) \Big\}}(\xi)\right)^2 d\xi\\
  & =  \int_{\mathbb{R}^n} \Big\{\sum_{j=N_1}^{N_2} v_j \int_{a_j}^{a_{j+1}}\partial_t \mathcal F\left(e^{-t(-\Delta)^\alpha} f \right)(\xi)dt  \Big\}^2 d\xi\\
 &\le
   C\norm{v}^2_{\ell^\infty(\mathbb Z)}
   \int_{\mathbb{R}^n} \Big\{\sum_{j=N_1}^{N_2} \Big|\int_{a_j}^{a_{j+1}}\partial_t \mathcal F\left(e^{-t(-\Delta)^\alpha} f\right)(\xi) dt\Big|  \Big\}^2 d\xi\\
     &=   C_v
   \int_{\mathbb{R}^n} \Big\{\sum_{j=N_1}^{N_2}\int_{a_j}^{a_{j+1}}\abs{\xi}^{2\alpha} e^{-t\abs{\xi}^{2\alpha}} \big|\mathcal F( f)(\xi)\big|dt  \Big\}^2 d\xi\\
   &\le    C_v
   \int_{\mathbb{R}^n} \Big\{\sum_{j=N_1}^{N_2} \int_{a_j}^{a_{j+1}}\abs{\xi}^{2\alpha} e^{-t\abs{\xi}^{2\alpha}} dt|\mathcal F( f)(\xi)|  \Big\}^2 d\xi\\ &\le  C_v
   \int_{\mathbb{R}^n} \Big\{\Big|\int_0^\infty\abs{\xi}^{2\alpha} e^{-t\abs{\xi}^{2\alpha}} dt\Big||\mathcal F( f)(\xi)|  \Big\}^2 d\xi\\
   & \le  C_{v,n} \|f\|_{L^2(\real^n)}^2.
\end{align*}
Then the proof of the theorem is complete.
\end{proof}

\begin{lem}\label{Lem:heatL}
There exists a constant $C>0$ depending on $n$ and $\alpha$ such that
\begin{itemize}
\item[(i)] \begin{equation*}\label{1}
0<{e^{-t(-\Delta)^\alpha} (x)}\leq C \frac t{(t^{\frac 1{2\alpha}}+|x|)^{n+2\alpha}},\quad x\in\real^n,~t>0,
\end{equation*}

\item[(ii)]
\begin{equation*}
 \abs{ \partial_t e^{-t(-\Delta)^\alpha} (x)}\leq C \frac 1{(t^{\frac 1{2\alpha}}+|x|)^{n+2\alpha}},\quad x\in\real^n,~t>0,
\end{equation*}
\item[(iii)]
\begin{equation*}
 \abs{ \nabla_x e^{-t(-\Delta)^\alpha} (x)}\leq C \frac 1{(t^{\frac 1{2\alpha}}+|x|)^{n+1}},\quad x\in\real^n,~t>0,
\end{equation*}
 and
 \item[(iv)]
\begin{equation*}
{ \abs{\partial_t \nabla_x e^{-t(-\Delta)^\alpha} (x)}\leq C \frac {1}{(t^{\frac 1{2\alpha}}+|x|)^{n+2\alpha+1}},\quad x\in\real^n,~t>0.}
\end{equation*}
 \end{itemize}
\end{lem}
\begin{proof}
For (i), it was proved in \cite[Lemma 5.4]{Grigo}. For the other estimations, we can get the proof  easily by using the results in \cite[Lemmas 2.1--2.2, Remark 2.1]{MYZ}.
\end{proof}

The following proposition  contains the size description of the kernel and the smoothness estimates that are required in the Calder\'on-Zygmund theory.
\begin{prop}\label{Thm:KernelEst}
Let $f\in L^p(\mathbb{R}^n), 1\le p \le \infty$. Then
$$T_Nf(x) = \int_{\mathbb{R}^n}  K_N(y)f(x-y)dy$$
with
\begin{equation*}\label{kernel}
  K_N(y) =  \sum_{j=N_1}^{N_2}v_j \left(e^{-a_{j+1}(-\Delta)^\alpha} (y)-e^{-a_{j}(-\Delta)^\alpha} (y)\right).
\end{equation*}
Moreover, there exists  constant $C>0$ depending  on $n, \alpha$ and $\norm{v}_{\ell^\infty(\mathbb Z)}$(not on $N$) such that, for any $y\neq 0,$
\begin{enumerate}[\indent i)]
  \item $\displaystyle   | K_N(y)|\leq \frac{C}{|y|^{n}}$,
  \item $\displaystyle   |\nabla_y  K_N(y)| \leq \frac{C}{|y|^{n+1}}$.
\end{enumerate}
\end{prop}

\begin{proof}
{\it i)}~Regarding  the size condition for the kernel, by Lemma \ref{Lem:heatL} we have
\begin{align*}
  | K_N (y)| &\leq  \sum_{j=N_1}^{N_2}\abs{v_j} \abs{e^{-a_{j+1}(-\Delta)^\alpha} (y)-e^{-a_{j}(-\Delta)^\alpha} (y)}\\
  &\le C_{n, v} \sum_{j=-\8}^{\8} \left|\int_{a_j}^{a_{j+1}}\partial_t e^{-t(-\Delta)^\alpha}(y) dt\right|\\
   &\le C_{n, v,\alpha}\int_{0}^{\infty}\frac 1{(t^{\frac 1{2\alpha}}+|y|)^{n+2\alpha}} dt  = C_{n, v,\alpha}{1 \over  {|y|^{n}}}.
\end{align*}

{\it ii)}
 With a similar argument as above in $i)$, by Lemma \ref{Lem:heatL}  we get
\begin{align*}
  |\nabla_y  K_N (y)| &\leq  \sum_{j=N_1}^{N_2}\abs{v_j} \abs{\nabla_y e^{-a_{j+1}(-\Delta)^\alpha}(y)-\nabla_y e^{-a_{j}(-\Delta)^\alpha}(y)}\\
  &\le C_{n, v} \sum_{j=-\8}^{\8}\int_{a_j}^{a_{j+1}}\left|\partial_t \nabla_y e^{-t(-\Delta)^\alpha}(y) \right|dt\\
   &\le C_{n, v,\alpha}\int_{0}^{\infty}\frac {1}{(t^{\frac 1{2\alpha}}+|y|)^{n+2\alpha+1}} dt  = C_{n, v,\alpha}{1 \over  {|y|^{n+1}}}.
\end{align*}
The proof of the  proposition is complete.
\end{proof}

Then, we have the following theorem about the uniform boundedness of $T_N$:

\begin{thm}\label{Thm:BMO} Let $\{a_j\}_{j\in \mathbb Z}$ be a positive increasing sequence and $\{T_N\}_{N=(N_1,N_2)}$ be  the operator $T_N$  defined in (\ref{Formu:FinSquareFun}). We have the following statements:
\begin{enumerate}[(a)]

\item for any $1<p<\infty$ and $\omega\in A_p$,  there exists a constant $C$ depending  on $n, p,\omega$, $\norm{v}_{l^\infty(\mathbb Z)}$ and $\alpha$ such that
 $$\norm{T_N f}_{L^p(\mathbb R^{n}, \omega)}\leq C\norm{f}_{L^p(\mathbb R^{n}, \omega)},$$
 for all functions $f\in L^p(\real^{n}, \omega).$
    \item for any  $\omega\in A_1$, there exists a constant $C$ depending  on $n, \omega$,  $\norm{v}_{\ell^\infty(\mathbb Z)}$  and  $\alpha$  such that
 $$\omega\left({\{x\in \real^{n}:\abs{T_N f(x)}>\sigma\}}\right) \le C\frac{1}{\sigma}\norm{f}_{L^1(\mathbb R^{n}, \omega)}, \quad \sigma>0,$$
for all functions $f\in L^1(\real^{n}, \omega).$

    \item  there exists a constant $C$ depending on $n$, $\norm{v}_{\ell^\infty(\mathbb Z)}$ and $\alpha$ such that
$$\norm{T_N f}_{BMO(\mathbb R^{n})}\leq C\norm{f}_{L^\infty(\mathbb R^{n})},$$
for all functions $f\in L^\infty(\real^{n}).$
\item  there exists a constant $C$ depending  on $n$, $\norm{v}_{\ell^\infty(\mathbb Z)}$ and $\alpha$ such that
     $$\norm{T_N f}_{BMO(\real^n)}\le C\norm{f}_{BMO(\real^n)}.$$

\end{enumerate}

The constants $C$ appeared above all are independent of $N.$
\end{thm}

\begin{proof}
 Previously, we have remarked  that the constants  in the $L^p$ boundedness only depend on the initial constant in $L^{p_0}(\mathbb{R}^n)$(in our case $p_0=2$), the size constant and smoothness constant of the kernel.  Hence the uniform boundedness of the operators $T_N$ in $L^p(\mathbb{R}^n)$ spaces is a direct consequence of the   Calder\'on-Zygmund theory. The finiteness of $T_N$ for functions in  $L^\infty(\real^n)$ is obvious, since for each $N$, $ K_N$ is an integrable function.
  On the other hand, if $f\in BMO(\real^n)$, we can proceed as follows. Let $B=B(x_0, r_0)$ and $B^*=B(x_0, 2r_0)$ with some $x_0\in \real^n$ and $r_0>0$. We decompose $f$ to  be
$$f=(f-f_B)\chi_{B^*}+(f-f_B)\chi_{(B^*)^c}+f_B=:f_1+f_2+f_3.$$
The function $f_1$ is integrable, hence  $T_N f_1(x)<\infty,$ $a.e.\ x\in \real^n.$ For $T_N f_2$, we note that, for any $x\in B$ and $t>0$,
\begin{align*}
&e^{-t(-\Delta)^{\alpha}} f_2(x)= \int_{\real^n} e^{-t(-\Delta)^{\alpha}}(x-y) f_2 (y)dy\\
& \le C \sum_{k=1}^\infty\int_{2^kr_0<|x_0-y|\le 2^{k+1}r_0}
           \frac t{(t^{\frac 1{2\alpha}}+|x-y|)^{n+2\alpha}}\abs{f(y)-f_B}dy\\
&\le Ct\sum_{k=1}^\infty{(2^kr_0)}^{-2\alpha}{1\over {(2^kr_0)}^{n}}\int_{|x_0-y|\le 2^{k+1}r_0} \abs{f(y)-f_B}dy\\
&\le  Ct  \sum_{k=1}^\infty{(2^kr_0)}^{-2\alpha}(1+2k)\norm{f}_{BMO(\real^n)}<\infty.
\end{align*}
So, $e^{-t(-\Delta)^{\alpha}} f_2(x)$ is finite for any $x\in B$ and $t>0.$ Since $T_N f_2(x)$ is a finite summation and $x_0, r_0$ is arbitrary, $T_N f_2(x)<\infty$ $a.e.$ $x\in \real^n.$
Finally we note that $T_N f_3(x)\equiv 0$, since $e^{{- a_j(-\Delta)^{\alpha}}} f_3\equiv f_B$ for any $j\in \mathbb Z.$
Hence, $T_N f(x)<\infty$ $a.e.$ $x\in \real^n.$ Then, by Propositions \ref{Thm:L2Estimate} and  \ref{Thm:KernelEst}  we  get the proof of part $(c)$ of   Theorem \ref{Thm:BMO}. To get $(d)$, since $T_N1=0$, the known arguments give the conclusion, see \cite{MTX}.
\end{proof}

\section{Boundedness of the maximal operator $T^*$}\label{Sec:maxi}

In this section, we will give the proof of Theorem \ref{Thm:PoissonLp} related to the boundedness of the maximal operator $T^*$. The next lemma,  parallel to  Proposition 3.2 in \cite{BLMMDT}(also Proposition 3.1 in \cite{ZMT}), shows that, without lost of generality, we may assume that
\begin{equation}\label{equ:lacunary}
1<\lambda \leq {a_{j+1} \over a_j}\leq \lambda^2, \quad j\in \mathbb Z.
\end{equation}

\begin{lem}\label{Prop:lacunary}
Given a $\lambda$-lacunary sequence $\{a_j\}_{j\in \mathbb Z}$ and a multiplying sequence $\{v_j\}_{j\in \mathbb Z}\in \ell^\infty(\mathbb Z)$, we can define a $\lambda$-lacunary sequence $\{\eta_j\}_{j\in \mathbb Z}$ and $\{\omega_j\}_{j\in \mathbb Z}\in \ell^\infty(\mathbb Z)$ verifying the following properties:
\begin{enumerate}[(i)]
\item $1<\lambda \leq \eta_{j+1}/\eta_j\leq \lambda^2,\quad \norm{\{\omega_j\}}_{\ell^\infty(\mathbb Z)}=\norm{\{v_j\}}_{\ell^\infty(\mathbb Z)}$.
\item For all $N=(N_1, N_2)$, there exists $N'=(N_1', N_2')$ with $T_N=\tilde{T}_{N'},$
where $\tilde{T}_{N'}$ is the operator defined in \eqref{Formu:FinSquareFun} for the new sequences $\{\eta_j\}_{j\in \ent}$ and $\{\omega_j\}_{j\in \ent}.$
\end{enumerate}
\end{lem}

\begin{proof} We  follow closely the ideas in the proof
of Proposition 3.2 in \cite{BLMMDT}. We include them here for completeness.

Let $\eta_0=a_0$, and let us construct $\eta_j$ for positive $j$ as follows(the argument for
negative $j$ is analogous). If $\lambda^2\ge {a_1/ a_0}\ge \lambda,$ define $\eta_1=a_1$. In the opposite case where
$a_1/a_0>\lambda^2,$  let $\eta_1=\lambda a_0.$ It verifies $\lambda^2\ge \eta_1/\eta_0=\lambda\ge \lambda.$ Further, $a_1/\eta_1\ge \lambda^2 a_0/\lambda a_0=\lambda.$  Again, if $a_1/\eta_1\le \lambda^2,$ then $\eta_2=a_1.$ If this is not the case, define $\eta_2=\lambda^2 a_0\le a_1$.  By
the same calculations as before, $\eta_0, \eta_1, \eta_2$ are part of a lacunary sequence satisfying \eqref{equ:lacunary}.
To continue the sequence, either $\eta_3=a_1$ (if $a_1/ \eta_2\le \lambda^2$) or $\eta_2=\lambda^3\eta_0$ (if $a_1/\eta_2>\lambda^2$). Since $\lambda>1,$ this process ends at some $j_0$ such that $\eta_{j_0}=a_1.$ The rest of the elements $\eta_j$
are built in the same way, as the original $a_k$ plus the necessary terms put in between two
consecutive $a_j$ to get \eqref{equ:lacunary}.
Let $J(j)=\{k:a_{j-1}<\eta_j\le a_j\}$, and $\omega_k=v_j$ if $k\in J(j)$. Then
\begin{equation*}
 v_j(e^{-a_{j+1} (-\Delta)^\alpha} f(x)-e^{-a_{j} (-\Delta)^\alpha} f(x))=\sum_{k\in J(j) }\omega_k(e^{-a_{k+1} (-\Delta)^\alpha} f(x)-e^{-a_{k} (-\Delta)^\alpha} f(x)).
\end{equation*}
If $M=(M_1, M_2)$ is the number such that $\eta_{M_2}=a_{N_2}$ and $\eta_{M_1-1}=a_{N_1-1}$, then we get
\begin{multline*}
 T_N f(x)=\sum_{j=N_1}^{N_2} v_j(e^{-a_{j+1} (-\Delta)^\alpha} f(x)-e^{-a_{j} (-\Delta)^\alpha} f(x))\\=\sum_{k=M_1}^{M_2} \omega_k(e^{-a_{k+1} (-\Delta)^\alpha} f(x)-e^{-a_{k} (-\Delta)^\alpha} f(x))= \tilde{T}_M f(x),
\end{multline*}
where $\tilde{T}_M$  is the operator defined in \eqref{Formu:FinSquareFun} related with sequences $\{\eta_k\}_{k\in \mathbb Z}$, $\{\omega_k\}_{k\in \mathbb Z}$ and $M=(M_1, M_2)$.
\end{proof}

This proposition allows us to assume in the rest of the article that the  lacunary sequences
 $\{a_j\}_{j\in \mathbb Z}$  satisfy \eqref{equ:lacunary} without saying it explicitly.

In order to prove Theorem \ref{Thm:PoissonLp} for the case of the fractional laplacian we shall need a Cotlar's type inequality to control the operator $T^*$.
Namely, we shall prove the following theorem:
\begin{thm}\label{Thm:Maximalcontrol}
For each $q\in (1, +\infty),$ there exists a constant $C$ depending  on $n, \norm{v}_{\ell^\infty(\mathbb Z)}$and $\lambda$  such that,  for every $x\in \real^n$ and every $M\in \mathbb Z^+$,
\begin{equation*}
T_M^*f(x)\le C\left\{\M(T_{(-M, M)} f)(x)+\M_q f(x)\right\},
\end{equation*}
where $$T_M^{*}f(x)=\sup_{-M\le N_1<N_2\le M}\abs{ T_N f(x)}$$ and
\begin{equation*}\label{Mq}
\M_qf(x)=\sup_{r>0} \left(\frac{1}{|B(x, r)|}\int_{B(x, r)}\abs{f(y)}^qdy\right)^{1\over q},\quad 1< q<\infty.\end{equation*}
\end{thm}

For the proof of this theorem we shall need the following lemma:

\begin{lem}\label{cotlar}  Let  $\{a_j\}_{j\in \mathbb Z}$  a $\lambda$-lacunary sequence and $\{v_j\}_{j\in \mathbb Z} \in \ell^\infty(\mathbb Z)$. Then
\begin{itemize}
\item[(i)] $\displaystyle \abs{\sum_{j=m}^{M}v_j \left(e^{-a_{j+1} (-\Delta)^\alpha}(x-y)- e^{-a_{j} (-\Delta)^\alpha}(x-y)\right) } \le { {C_{v, \lambda,\alpha, n}} \over a_m^{n/2\alpha}}, $
\

\item[(ii)] if $k\ge m$ and $z,y \in\mathbb{R}^n$ with $|z-y|\ge  a_k^{1/2\alpha}$, then \begin{align*}
\Big|\sum_{j=-M}^{m-1}v_j \left(e^{-a_{j+1}(-\Delta)^\alpha}(z-y)-e^{-a_{j}(-\Delta)^\alpha}(z-y)  \right) 	 \Big|\, \le C_{\lambda, v,\alpha, n}\frac1{a_k^{n/2\alpha}}\lambda^{-(k-m+1)}.
\end{align*}

\end{itemize}

\end{lem}
\begin{proof}
By the mean value theorem and Lemma \ref{Lem:heatL},  there exists  $a_j\le \xi_j\le a_{j+1}$ such that
\begin{align*}
&\abs{\sum_{j=m}^{M}v_j \left(e^{-a_{j+1}(-\Delta)^\alpha}(x-y)- e^{-a_{j}(-\Delta)^\alpha}(x,y) \right) }\\
&\le  C\norm{v}_{l^\infty(\mathbb Z)} \sum_{j=m}^{M} (a_{j+1}-a_j)\abs{\partial_t e^{-t(-\Delta)^\alpha}(x,y)\big|_{t=\xi_j}}\\
  &\le C_{v}  \sum_{j=m}^{M} (a_{j+1}-a_j) \abs{\frac 1{(\xi_j^{\frac{1}{2\alpha}}+|x-y|)^{n+2\alpha}}} \\
  &\le C_{v}  \sum_{j=m}^{M} (\lambda^2-1)a_j ^{-{n\over 2\alpha}}
   \le  C_{v, \lambda} \frac1{a_m^{n/2\alpha}} \sum_{j=m}^{M}  \frac1{\lambda^{n(j-m)/2\alpha}}   \le C_{v, \lambda,\alpha, n} {1\over a_m^{n/2\alpha}},
\end{align*}
where we have used $\displaystyle \lambda\le {a_{j+1}\over a_j}\le \lambda^2.$

Now we shall prove (ii). By the mean value theorem, there exist $a_j\le \xi_j\le a_{j+1}$ such that
\begin{align*}
& \abs{\sum_{j=-M}^{m-1}v_j \left(e^{-a_{j+1}(-\Delta)^\alpha}(z-y)-e^{-a_{j}(-\Delta)^\alpha}(z-y) \right) }\\
&\le
  C\norm{v}_{l^\infty(\mathbb Z)} \sum_{j=m}^{M} (a_{j+1}-a_j)\abs{\partial_t e^{-t(-\Delta)^\alpha}(x,y)\big|_{t=\xi_j}}\\&\le C_{v,\alpha, n}\sum_{j=-M}^{m-1}(\lambda^2-1)a_j \frac 1{(\xi_j^{ 1/{2\alpha}}+|x-y|)^{n+2\alpha}} \\
  &\le C_{\lambda, v,\alpha, n}\sum_{j=-M}^{m-1}{a_j \over (\xi_j^{ 1/{2\alpha}}+|x-y|)^{n+2\alpha}} \\
&\le C_{\lambda, v,\alpha, n} \sum_{j=-M}^{m-1} {\frac{ a_j }{a_k }}\left({1\over a_k^{n/2\alpha}}\right)
\le C_{\lambda, v,\alpha, n}  {1\over a_k^{n/2\alpha}}\lambda^{-(k-m+1)},
\end{align*}
where we have used that $k \ge m.$
\end{proof}

\begin{proof}[Proof of Theorem \ref{Thm:Maximalcontrol}]
 Observe that, for any $x_0\in \real^n$ and  $N=(N_1, N_2),$
$$T_N f(x_0)=T_{(N_1, M)} f(x_0)-T_{(N_2+1, M)} f(x_0),$$
with $-M\le N_1<N_2\le  M.$
Then, it suffices to estimate $\abs{T_{(m, M)} f(x_0)}$ for $\abs{m}\le M$ with constants independent of $m$ and $M.$ Denote $B_k=B(x_0, a_k^{1/2\alpha})$ for each $k\in \mathbb N$.
Let us split $f$ as
\begin{align*}
f&=f\chi_{B_m}+f\chi_{B_m^c}=:f_1+f_2.
\end{align*}
 Then,  we have
\begin{align*}
\abs{T_{(m,M)} f(x_0)}&\le \abs{ T_{(m,M)} f_1(x_0)}+\abs{T_{(m,M)} f_2(x_0)}\\
&=: I+II.
\end{align*}
For $I$, by Lemma \ref{cotlar} (i), we have
\begin{align*}
I&=\abs{T_{(m,M)} f_1(x_0)}=\abs{\int_{\real^n} \sum_{j=m}^{M}v_j \left(e^{-a_{j+1}(-\Delta)^\alpha}(x_0,y)- e^{-a_{j}(-\Delta)^\alpha}(x_0,y)\right)  f_1(y)dy}\\
&\le C_{n, v, \lambda, \alpha}  {1\over \abs{a^{n/2}_m}} \int_{\real^n}  \abs{f_1(y)}dy \le C_{v, \lambda, n, \alpha}\M f(x_0).
\end{align*}
For part $II$,
\begin{align*}\label{equ:II}
II&=\abs{T_{(m,M)} f_2(x_0)}=\frac{2^{n/2}}{a_{m-1}^{n/2\alpha}}\int_{B(x_0, {1\over 2}a_{m-1}^{1/2\alpha})} \abs{T_{(m,M)} f_2(x_0)}dz \\
&\le \frac{2^{n/2}}{a_{m-1}^{n/2\alpha}}\int_{B(x_0, {1\over 2}a_{m-1}^{1/2\alpha })}\abs{T_{(-M,M)} f(z)}dz +\frac{2^{n/2}}{a_{m-1}^{n/2\alpha}}\int_{B(x_0, {1\over 2}a_{m-1}^{1/2\alpha})}\abs{T_{(-M,M)} f_1(z)}dz\\
&\quad +\frac{2^{n/2}}{a_{m-1}^{n/2\alpha}}\int_{B(x_0, {1\over 2}a_{m-1}^{1/2\alpha})}\abs{T_{(m,M)} f_2(z)-T_{(m,M)} f_2(x_0)}dz\\
&\quad +\frac{2^{n/2}}{a_{m-1}^{n/2\alpha}}\int_{B(x_0, {1\over 2}a_{m-1}^{1/2\alpha})}\abs{T_{(-M,m-1)} f_2(z)}dz\\
&=:A_1+A_2+A_3+A_4.
\end{align*}
(If $m+1=-M$, we understand that $A_4=0$.) It is clear that
\begin{equation*}
A_1\le  \M (T_{(-M,M)} f)(x_0).
\end{equation*}
For $A_2,$ by the uniform boundedness of $T_{N}$, we get
\begin{multline*}
A_2\le \left(\frac{2^{n/2}}{a_{m-1}^{n/2\alpha}}\int_{B_{m-1}}\abs{T_{(-M,M)} f_1(z)}^qdz\right)^{1/q}\le C\left(\frac{1}{a_{m-1}^{n/2\alpha}}\int_{\mathbb{R}^n}\abs{f_1(z)}^qdz\right)^{1/q}\\
=C\left(\frac{1}{a_{m-1}^{n/2\alpha}}\int_{B_{m}}\abs{f(z)}^qdz\right)^{1/q}\le C\left(\frac{\lambda^{n/2\alpha}}{a_{m}^{n/2\alpha}}\int_{B_{m}}\abs{f(z)}^qdz\right)^{1/q}\le C\M_q f(x_0).
\end{multline*}
For the third term $A_3$, with $z\in B(x_0, {1\over 2}a_{m-1}^{1/2\alpha})$, we have
\begin{align*}
&\abs{T_{(m,M)} f_2(z)-T_{(m,M)} f_2(x_0)}\\
&=\abs{\int_{B_{m}^c} K_{(m,M)}(z- y)f(y)dy-\int_{B_m^c}  K_{(m,M)}(x_0- y)f(y)dy }\\
&\le \int_{B_m^c} \abs{K_{(m,M)}(z- y)-K_{(m,M)}(x_0- y)}\abs{f(y)}dy\\
&=\sum_{j=1}^{+\infty}\int_{B_{2^jm} \setminus B_{2^{j-1}m}} \abs{K_{(m,M)}(z- y)-K_{(m,M)}(x_0- y)}\abs{f(y)}dy,
\end{align*}
where $B_{2^jm}=B(x_0, 2^ja_m^{1/ 2\alpha})$ for any $j\ge 1.$

By the mean value theorem, we know that  there exists $\xi$ on the segment  $\overline{x_0z}$ such that
\begin{align*}
&\abs{K_{m,M}(z- y)-K_{m,M}(x_0- y)}\le   \abs{\bigtriangledown_\xi K_{m,M}(\xi- y)}\abs{z-x_0}\\
&\le C \frac{\abs{z-x_0}}{\abs{\xi-y}^{n+1}} \le C\frac 1{2^{j}}\cdot\frac{a_{m-1}^{1/2\alpha}}{a_{m}^{1/2\alpha}}\cdot {1\over |2^ja_{m}|^{n/2\alpha}},
\end{align*}
where we have used that in each summand, $y \in B_{2^jm} \setminus B_{2^{j-1}m}$. Hence
\begin{multline*}
\abs{T_{m,M} f_2(z)-T_{m,M} f_2(x_0)} \le C\sum_{j=1}^{+\infty}\frac 1{2^{j}}\frac{a_{m-1}^{1/2\alpha}}{a_m^{1/2\alpha}}{1\over |2^ja_{m}|^{n/2\alpha}}\int_{B_{2^jm}} \abs{f(y)}dy\\
\le C \M f(x_0)\frac{a_{m-1}^{1/2\alpha}}{a_m^{1/2\alpha}}\sum_{j=1}^{+\infty}\frac 1{2^{j}}\le C \M f(x_0).
\end{multline*}
Then,
\begin{align*}
A_3=\frac{2^{n/2}}{a_{m-1}^{n/2\alpha}}\int_{B_{m-1}}\abs{T_{m,M} f_2(z)-T_{m,M} f_2(x_0)}dz \le C\M f(x_0).
\end{align*}
For the latest one,  $A_4,$  we have
\begin{multline*}
A_4=\frac{2^{n/2}}{a_{m-1}^{n/2\alpha}}\int_{B(x_0, {1\over 2}a_{m-1}^{1/2\alpha})}\abs{ T_{(-M,m-1)} f_2(z)}dz\\
\le \frac{2^{n/2}}{a_{m-1}^{n/2\alpha}}\int_{B(x_0, {1\over 2}a_{m-1}^{1/2\alpha})}\int_{B_m^c}\abs{ K_{(-M,m-1)}(z- y) f(y)}dydz.
\end{multline*}
Then, we consider the inner integral appeared in the above inequalities first. Since $z\in B(x_0, {1\over 2}a_{m-1}^{1/2\alpha})$,  $y\in B_m^c$ and the sequence $\{a_j\}_{j\in \mathbb Z}$ is $\lambda$-lacunary sequence, we have  $\abs{z-y}\sim \abs{y-x_0}.$
From this and by Lemma \ref{cotlar} (ii), we get
\begin{align*}
&\int_{B_m^c}\abs{K_{(-M,m-1)}(z- y) f(y)}dy\\
   &=\sum_{k=m}^{+\infty}\int_{B_{k+1}\setminus B_{k}} \abs{\sum_{j=-M}^{m-1}v_j \left( e^{-a_{j+1}(-\Delta)^\alpha}(z- y)-e^{-a_{j}(-\Delta)^\alpha}(z- y)\right) f(y)}dy\\
&\le C_{\lambda, v,\alpha, n} \sum_{k=m}^{+\infty}\lambda^{-(k-m+1)}\left({1\over a_k^{n/2\alpha}}\int_{B_{k+1}\setminus B_{k}} \abs{f(y)}dy\right)\\
 &\le C_{\lambda, v,\alpha, n} \M f(x_0)\sum_{k=m}^{+\infty}\lambda^{-(k-m+1)}\\
 &\le C_{\lambda, v,\alpha, n}\M f(x_0).
\end{align*}
Hence,
$$A_4\le C\M f(x_0).$$
Combining the estimates above for $A_1, A_2, A_3$ and $A_4$, we get
$$II\le \M ( T_{(-M,M)} f)(x_0)+C \M_q f(x_0).$$
And then we have
$$ \abs{ T_{(m,M)} f(x_0)}\le C\left( \M (T_{(-M,M)} f)(x_0)+\M_q f(x_0)\right). $$
As the constants $C$ appeared above all only depend on $\norm{v}_{l^\infty(\mathbb Z)}$, $\lambda$, $\alpha$ and $n$, we have  proved that
$$ T^{*}_M f(x_0)\le C\left\{\M( T_{(-M, M)} f)(x_0)+\M_q f(x_0)\right\}.$$ This complete the proof of the theorem.
\end{proof}

Now, we can give the proof of Theorem \ref{Thm:PoissonLp}.

\begin{proof}[Proof of Theorem \ref{Thm:PoissonLp}] {Given $\omega \in A_p$, we choose $1<q<p $ such that $\omega \in A_{p/q}$.
Then it is well known that, the maximal operators $\M$ and $ \M_q$ are bounded on $L^p(\real^n, \omega),$ see \cite{Duo}.}  On the other hand, since the operators $T_N$ are uniformly bounded in $L^p(\real^n, \omega)$ with $\omega \in A_p$, we have
\begin{align*}
\norm{T_M^*f}_{L^p(\real^n, \omega)}&\le C\left(\norm{\M (T_{(-M, M)} f)}_{L^p(\real^n, \omega)}+\norm{\M_q f}_{L^p(\real^n, \omega)}\right)\\&\le C\left(\norm{T_{(-M, M)} f}_{L^p(\real^n, \omega)}+\norm{f}_{L^p(\real^n, \omega)}\right)\le C\norm{f}_{L^p(\real^n, \omega)}.
\end{align*}
Note that the constants $C$ appeared above do not depend on $M$. Consequently, letting $M$ increase to infinity, we get the proof of the $L^p$ boundedness of the maximal operator $T^*.$
This completes the proof of part $(a)$ of the theorem.

In order to prove $(b)$, we consider the $\ell^\infty(\mathbb Z^2)$-valued operator
$\mathcal{T}f(x) = \{  T_N f(x) \}_{N\in \mathbb Z^2}$. Since  $\|\mathcal{T}f(x) \|_{\ell^\infty(\mathbb Z^2)}= T^*f(x)$,   by using $(a)$ we know that the operator $\mathcal{T}$ is bounded from $L^p(\real^n, \omega) $ into $L^p_{\ell^\infty(\mathbb Z^2)}(\mathbb{R}^n, \omega) $, for every $1<p<\infty$ and $\omega \in A_p$. The kernel of the operator $\mathcal{T}$ is given by $\mathcal{K}(x) = \{ K_N(x)\} _{N\in \mathbb Z^2}$. Therefore, by the vector valued Calder\'on-Zygmund theory, the operator $\mathcal{T}$ is bounded from $L^1(\mathbb{R}^n, \omega)$ into weak-$L^1_{\ell^\infty(\mathbb Z^2)}(\mathbb{R}^n, \omega)$ for $\omega \in A_1$. Hence, as $\|\mathcal{T}f(x) \|_{\ell^\infty(\mathbb Z^2)}= T^*f(x)$, we get the proof of  $(b)$.

For $(c)$ and $(d)$, we shall first prove that, if $f\in BMO(\mathbb R^n)$ and there exists $x_0\in \mathbb R^n$ such that $T^*f(x_0)<\infty,$ then $T^*f(x)<\infty$ for $a.e.$ $x\in \real^n.$
Set $B=B(x_0, 4\abs{x-x_0})$ with $x\neq x_0$. And we decompose $f$ to be
\begin{equation*}
f=(f-f_B)\chi_B+(f-f_B)\chi_{B^c}+f_B=:f_1+f_2+f_3.
\end{equation*}
Note that $T^*$ is $L^p$ bounded for any $1<p<\infty.$ Then $T^*f_1(x)<\infty$, because $f_1\in L^p(\mathbb R^n)$, for any $1<p<\infty.$ And  $T^* f_3=0$, since $e^{-a_j(-\Delta)^\alpha}f_3=f_3$ for any $j\in \mathbb Z.$
On the other hand by the smoothness properties of the kernel, we have
\begin{align*}
&\Big|T_N f_2(x)-T_N f_2(x_0)\Big|=\Big| \int_{B^c}\left(K_N(x- y) - K_N (x_0- y)\right)f_2(y)dy\Big|\\
&\le C \int_{B^c}  \frac{\abs{x-x_0}}{\abs{y-x_0}^{n+1}}\abs{f(y)-f_B}dy  \\
&\le C \sum_{k=1}^{+\infty}{ |x-x_0|} \int_{2^{k} B\setminus 2^{k-1}B}  {\abs{f(y)-f_B}\over |y-x_0|^{n+1}}dy\\
&\le C \sum_{k=1}^{+\infty}{ |x-x_0|\over (2^{k+1}|x-x_0|)^{n+1}} \int_{2^kB}  {\abs{f(y)-f_B}}dy \\ &\le C \sum_{k=1}^{+\infty}2^{-(k+1)}{ 1\over \abs{2^{k}B}} \int_{2^{k}B} \left(  {\abs{f(y)-f_{2^kB}}}+\sum_{l=1}^{k}\abs{f_{2^lB}-f_{2^{l-1}B}}\right)dy\\
&\le C\sum_{k=1}^{+\infty}2^{-(k+1)}{ 1\over \abs{2^{k}B}} \int_{2^{k}B} \left(  {\abs{f(y)-f_{2^kB}}}+2k\norm{f}_{BMO(\real^n)}\right)dy\\
&\le C\sum_{k=1}^{+\infty}2^{-(k+1)}{(1+2k)\norm{f}_{BMO(\real^n)}}\\
&\le  C\norm{f}_{BMO(\real^n)},
 \end{align*}
where $2^kB=B(x_0, 2^{k}\cdot 4|x-x_0|)$ for any $k\in \mathbb N.$
 Hence
  \begin{align*}
\norm{T_N f_2(x)-T_N f_2(x_0)}_{l^\infty(\mathbb Z^2)} \le C\norm{f}_{BMO(\mathbb R^n)},
\end{align*}
and therefore $ T^*f(x) = \norm{T_N f(x)}_{l^\infty(\mathbb Z^2)}  \le C < \infty.$

Finally,  the  estimate (\ref{sharp}) can be  proved in a parallel way to the proof of part $(b)$ of Theorem \ref{Thm:BMO}, since $\mathcal{T}1(x) = \{T_N1(x)\}= 0$(also see \cite{MTX}).
\end{proof}

From Theorem \ref{Thm:PoissonLp}, we can get the following consequence:

\begin{thm}\label{Thm:ae}
\begin{enumerate}[(a)]
    \item If $1<p<\infty$ and $\omega\in A_p$, then $T_N f$ converges {\it{a.e.}} and  in $L^p(\mathbb R^n, \omega)$ norms for all $f\in L^p(\mathbb R^n, \omega)$ as $N=(N_1,N_2)$ tends to $(-\infty, +\infty).$
   \item If $p=1$  and $\omega\in A_1$, then $T_N f$ converges {\it{a.e.}} and in measure for all $f\in L^1(\mathbb R^n, \omega)$ as $N=(N_1,N_2)$ tends to $(-\infty, +\infty).$
\end{enumerate}
\end{thm}

\begin{proof}
First, we shall see that if $\varphi$ is a test function, then $T_N \varphi(x)$ converges for all $x\in \real^n$.  In order to prove this, it is enough to see that for any  $(L,M)$ with $0<L<M$,  the  series \begin{equation*} A= \sum_{j=L}^M v_j ( e^{-a_{j+1}(-\Delta)^\alpha} \varphi(x) - e^{-a_j(-\Delta)^\alpha} \varphi(x))$$   and $$  B= \sum_{j=-M}^{-L} v_j ( e^{-a_{j+1}(-\Delta)^\alpha} \varphi(x) - e^{-a_j(-\Delta)^\alpha} \varphi(x))
 \end{equation*}
converge to zero, when $L, M\rightarrow +\infty$.
By Lemma \ref{Lem:heatL}, we have
\begin{align*}
|A|  &=\abs{\sum_{j=L}^M v_j \int_{\real^n}( e^{-a_{j+1}(-\Delta)^\alpha}(x-y) - e^{-a_j(-\Delta)^\alpha}(x-y)) \varphi(y)dy}\\
 & \le   C\int_{\real^n}\Big(\sum_{j=L}^{M} \int_{a_j}^{a_j+1} \frac{1}{(t^{1/2\alpha}+|x-y|)^{n+2\alpha}} dt\Big) |\varphi(x-y)|dy\\
&\le C   \big(a_{M+1}^{-{n\over 2\alpha}}-a_{L}^{-{n\over 2\alpha}}\big) \|\varphi\|_{L^1(\real^n)} \longrightarrow 0, \quad \hbox{as}\ { L,M \to +\infty}.
\end{align*}
On the other hand, since $\displaystyle \int_{\real^n}( e^{-a_{j+1}(-\Delta)^\alpha}(x-y) - e^{-a_j(-\Delta)^\alpha}(x-y))dy=0$ for any $j\in \mathbb Z,$   we can write
\begin{align*}
B&=\int_{\real^n} \sum_{j=-M}^{-L}v_j ( e^{-a_{j+1}(-\Delta)^\alpha}(y) - e^{-a_j(-\Delta)^\alpha}(y)) \left(\varphi(x-y)-\varphi(x)\right)dy.
\end{align*}
And, then proceeding as in the case $A$, and by using the fact that $\varphi$ is a test function,  we have
\begin{align*}
|B|&\le C \norm{v}_{\ell^\infty(\mathbb Z)} \int_{\real^n} \sum_{j=-M}^{-L} \abs{ e^{-a_{j+1}(-\Delta)^\alpha}(y) - e^{-a_j(-\Delta)^\alpha}(y) }\abs{ \varphi(x-y)-\varphi(x)}dy\\
&\le C \norm{\nabla \varphi}_{L^\infty(\real^n)}\sum_{j=-M}^{-L} \int_{\real^n}  \frac {a_{j+1}|y|} {(a_j^{1/2\alpha}+|y|)^{n+2\alpha}} dy\\
&\le  C \norm{\nabla \varphi}_{L^\infty(\real^n)}\sum_{j=-M}^{-L} a_j^{1\over 2\alpha}\int_{\real^n}  \frac {|y|}{a_{j}^{1\over 2\alpha}}
\frac 1 {(1+{|y|\over a_j^{1/2\alpha}} )^{n+2\alpha}} {dy\over a_j^{n/2\alpha}} \\
&\le C_{n,   \alpha}\norm{\nabla \varphi}_{L^\infty(\real^n)} a_{-L}^{1/2\alpha}\sum_{j=-M}^{-L} {a_j^{1/2\alpha}\over  a_{-L}^{1/2\alpha}}\\
&\le C_{n,   \alpha}\norm{\nabla \varphi}_{L^\infty(\real^n)} { \lambda^{1/2\alpha}\over  \lambda^{1/2\alpha}-1} a_{-L}^{1/2\alpha}\longrightarrow 0, \quad \hbox{as}\ { L,M \to +\infty},
\end{align*}
where we have used the assumption $ 1/2<\alpha<1$ to make the integral convergent.

For the case $0<\alpha< {1\over 2}$,       we can write
\begin{align*}
B
& =\int_{B(0,1)} \sum_{j=-M}^{-L}v_j ( e^{-a_{j+1}(-\Delta)^\alpha}(y) - e^{-a_j(-\Delta)^\alpha}(y)) \left(\varphi(x-y)-\varphi(x)\right)dy\\
&\quad +\int_{B(0,1)^c} \sum_{j=-M}^{-L}v_j ( e^{-a_{j+1}(-\Delta)^\alpha}(y) - e^{-a_j(-\Delta)^\alpha}(y)) \left(\varphi(x-y)-\varphi(x)\right)dy\\
&:=B_1+B_2.
\end{align*}
For $B_1$,  we have
\begin{align*}
|B_1|&\le C \norm{\nabla \varphi}_{L^\infty(\real^n)}\sum_{j=-M}^{-L} \int_{B(0,1)}  \frac {a_{j+1}|y|} {(a_j^{1/2\alpha}+|y|)^{n+2\alpha}} dy\\
&\le  C \norm{\nabla \varphi}_{L^\infty(\real^n)}\sum_{j=-M}^{-L} a_j^{1\over 2\alpha}\int_{B(0,1)}  \frac {|y|}{a_{j}^{1\over 2\alpha}}
\frac 1 {(1+{|y|\over a_j^{1/2\alpha}} )^{n+2\alpha}} {dy\over a_j^{n/2\alpha}} \\
&\le C \norm{\nabla \varphi}_{L^\infty(\real^n)} \sum_{j=-M}^{-L} a_j^{1/2\alpha}\int_0^{a_j^{-{1 / 2\alpha}}}{r^n\over {(1+r)^{n+2\alpha}}}dr\\
&\le C \norm{\nabla \varphi}_{L^\infty(\real^n)} a_{-L}\sum_{j=-M}^{-L} {a_j \over  a_{-L} }\\
&\le C \norm{\nabla \varphi}_{L^\infty(\real^n)} { \lambda \over  \lambda -1} a_{-L} \longrightarrow 0, \quad \hbox{as}\ { L,M \to +\infty}.
\end{align*}
For $B_2,$  we have
\begin{align*}
|B_2|&\le C \norm{ \varphi}_{L^\infty(\real^n)}\sum_{j=-M}^{-L} \int_{B(0,1)^c}  \frac {a_{j+1} } {(a_j^{1/2\alpha}+|y|)^{n+2\alpha}} dy\\
&\le  C \norm{  \varphi}_{L^\infty(\real^n)}\sum_{j=-M}^{-L} \int_{B(0,1)^c}
\frac 1 {(1+{|y|\over a_j^{1/2\alpha}} )^{n+2\alpha}} {dy\over a_j^{n/2\alpha}} \\
&\le C \norm{  \varphi}_{L^\infty(\real^n)} \sum_{j=-M}^{-L} \int^{+\infty}_ {a_j^{-{1 / 2\alpha}}}r^{-2\alpha-1}dr\\
&\le C \norm{  \varphi}_{L^\infty(\real^n)} a_{-L}\sum_{j=-M}^{-L} {a_j \over  a_{-L} }\\
&\le C \norm{  \varphi}_{L^\infty(\real^n)} { \lambda \over  \lambda -1} a_{-L} \longrightarrow 0, \quad \hbox{as}\ { L,M \to +\infty}.
\end{align*}
So, when $0<\alpha<1/2$, we proved that $|B|\rightarrow 0,$ as $L,M\rightarrow +\infty.$ And for the case $\alpha={1\over 2},  $  it has been proved in \cite{Zhang}. Hence, we proved that $T_N \varphi(x)$ converges for all $x\in \real^n$ with $\varphi$ being a test function.

As the set of test functions is dense in $L^p(\mathbb{R}^n)$, by Theorem \ref{Thm:PoissonLp} we get the $a.e.$ convergence for any function in $L^p(\mathbb{R}^n)$. Analogously, since $L^p(\mathbb{R}^n) \cap L^p(\mathbb{R}^n, \omega) $ is dense in $L^p(\mathbb{R}^n, \omega)$, we get the $a.e.$ convergence for functions in $L^p(\mathbb{R}^n, \omega) $ with $1\le p<\infty$. By using the dominated convergence theorem,  we can prove the convergence in $L^p(\mathbb{R}^n, \omega)$ norm for $1<p<\infty$, and also in measure.
\end{proof}

\section{Proof of the local growth of the maximal operator $T^*$}

In this section, we will give the proof of the local growth of the maximal operator $T^*.$
\begin{proof}[Proof of Theorem \ref{Thm:GrothLinfinity}]
We will prove it only in the case $1<p<\infty.$ For the case $p=1$ and $p=\infty$, the proof is similar and easier. Since $2r<1,$ we know that $B\backslash B_{2r}\neq \emptyset.$ Let $f(x)=f_1(x)+f_2(x)$, where $f_1(x)=f(x)\chi_{B_{2r}}(x)$ and $f_2(x)=f(x)\chi_{B\backslash B_{2r}}(x)$. Then
$$\abs{T^* f(x)}\le \abs{T^* f_1(x)}+\abs{T^* f_2(x)}.$$
By Theorem \ref{Thm:PoissonLp}, we have
\begin{multline*}
\frac{1}{|B_r|} \int_{B_r} \abs{T^* f_1 (x)} dx\le \left(\frac{1}{|B_r|} \int_{B_r} \abs{T^* f_1 (x)}^2 dx\right)^{1/2}
\\ \le C \left(\frac{1}{|B_r|} \int_{\mathbb{R}^n}\abs{ f_1 (x)}^2 dx\right)^{1/2}\le C\norm{f}_{L^\infty(\mathbb{R}^n)}.
\end{multline*}
On the other hand,  applying  H\"older's inequality on the integers and on $\mathbb{R}^n$, Fubini's Theorem and Lemma \ref{Lem:heatL}, for $1< p < \infty$ and any $N=(N_1, N_2)$, we have
\begin{align}\label{equ:Palpha}
&\abs{\sum_{j=N_1}^{N_2}v_j\left(e^{-a_{j+1}(-\Delta)^\alpha}f_2(x)- e^{-a_j(-\Delta)^\alpha}f_2(x)\right)}\nonumber\\
&\le C\sum_{j=N_1}^{N_2} \abs{v_j \int_{\real^n}\left(e^{-a_{j+1}(-\Delta)^\alpha}(y) -e^{-a_{j}(-\Delta)^\alpha}(y)\right)  f_2(x-y)~ dy}\nonumber\\
&\le C\norm{v}_{l^p(\mathbb Z)}\left( \sum_{j=N_1}^{N_2}\left(\int_{\real^n}\abs{e^{-a_{j+1}(-\Delta)^\alpha}(y) -e^{-a_{j}(-\Delta)^\alpha}(y)} \abs{f_2(x-y)}~ dy\right)^{p'}\right)^{1/p'} \nonumber \\
&\le C\norm{v}_{l^p(\mathbb Z)}\Big(\sum_{j=N_1}^{N_2} \Big\{\int_{\real^n}\abs{e^{-a_{j+1}(-\Delta)^\alpha}(y) -e^{-a_{j}(-\Delta)^\alpha}(y)} \abs{f_2(x-y)}^{p'}~ dy\Big\}  \\
 & \quad \quad  \times  \Big\{\int_{\real^n}\abs{e^{-a_{j+1}(-\Delta)^\alpha}(y) -e^{-a_{j}(-\Delta)^\alpha}(y) }~ dy  \Big\}^{p'/p} \Big)^{1/p'}
\nonumber \\
&\le C \norm{v}_{l^p(\mathbb Z)}\left(\sum_{j=N_1}^{N_2} \int_{\real^n}\abs{e^{-a_{j+1}(-\Delta)^\alpha}(y) -e^{-a_{j}(-\Delta)^\alpha}(y)  } \abs{f_2(x-y)}^{p'}~ dy   \right)^{1/p'}
\nonumber \\
&\le C \norm{v}_{l^p(\mathbb Z)}\left(\int_{\real^n}\sum_{j=-\infty}^{+\infty} \abs{e^{-a_{j+1}(-\Delta)^\alpha}(y) -e^{-a_{j}(-\Delta)^\alpha}(y)} \abs{f_2(x-y)}^{p'}~ dy  \right)^{1/p'}
\nonumber \\ \nonumber
&\le C \norm{v}_{l^p(\mathbb Z)}\left(\int_{\real^n}\left(\int_0^{+\infty}\frac t{(t^{1/2\alpha}+|y|)^{n+2\alpha}}dt\right) \abs{f_2(x-y)}^{p'}~ dy  \right)^{1/p'}
\nonumber \\ \nonumber
&\le C \norm{v}_{l^p(\mathbb Z)}\left(\int_{\real^n} \frac1{|y|^n} \abs{f_2(x-y)}^{p'}~ dy  \right)^{1/p'}.
\end{align}
Hence
\begin{align*}
\frac{1}{|B_r|} \int_{B_r} \abs{T^* f_2(x)}dx &\le  C\frac{1}{|B_r|} \int_{B_r} \left(\int_{\real^n} \frac1{|y|^n} \abs{f_2(x-y)}^{p'}~ dy  \right)^{1/p'}dx \\
&=  C\frac{1}{|B_r|} \int_{B_r} \left(\int_{\real^n} \frac1{|x-y|^n} \abs{f_2(y)}^{p'}~ dy  \right)^{1/p'}dx \\
&\le C\frac{\norm{f}_{L^\infty(\real^n)}}{|B_r|} \int_{B_r} \left(\int_{r\le |x-y| \le 2} \frac1{|x-y|^n} ~ dy \right)^{1/p'}dx \\
 &\sim \Big(\log\frac{2}{r}\Big)^{1/p'}\norm{f}_{L^\infty(\real^n)},
\end{align*}
where we have used the fact  $y\in B\backslash B_{2r}$ and $x\in B_r$.
Therefore we arrive to  $$\frac{1}{|B_r|} \int_{B_r} \abs{T^* f(x)}dx\le C\left(1+\Big(\log\frac{2}{r}\Big)^{1/p'}\right)\norm{f}_{L^\infty(\real^n)}  \le C\Big(\log\frac{2}{r}\Big)^{1/p'}\norm{f}_{L^\infty(\real^n)}.$$
 Then we get the proof of Theorem  \ref{Thm:GrothLinfinity}.

 \end{proof}

\vspace{2cm}

\end{document}